\documentclass[11pt,twoside,reqno]{amsart}

\usepackage{microtype}
\usepackage[OT1]{fontenc}
\usepackage{type1cm}
\usepackage{amssymb}
\usepackage{enumerate}
\usepackage{xcolor}
\usepackage{mathrsfs}
\usepackage{dsfont}
\usepackage{tikz}
\usetikzlibrary{calc,patterns,angles,quotes,arrows}
\usepackage{verbatim}

\usepackage{cite}

\usepackage[left=2.3cm,top=3cm,right=2.3cm]{geometry}
\numberwithin{equation}{section}

\theoremstyle{plain}
\newtheorem{theorem}{Theorem}[section]

\newtheorem{corollary}[theorem]{Corollary}
\newtheorem{proposition}[theorem]{Proposition}

\newtheorem{lemma}[theorem]{Lemma}

\theoremstyle{remark}
\newtheorem{remark}[theorem]{Remark}

\newtheorem{example}[theorem]{Example}
\newtheorem*{ack}{Acknowledgement}

\theoremstyle{definition}
\newtheorem{definition}[theorem]{Definition}

\newcommand{\GG}{\mathcal{G}}
\newcommand{\HH}{\mathcal{H}}

\newcommand{\R}{\mathbb{R}}

\newcommand{\N}{\mathbb{N}}
\newcommand{\PPP}{\mathbb{P}}

\newcommand{\iii}{\mathtt{i}}
\newcommand{\jjj}{\mathtt{j}}

\newcommand{\fii}{\varphi}
\newcommand{\roo}{\varrho}

\DeclareMathOperator{\dimb}{dim_B}

\DeclareMathOperator{\dimh}{dim_H}

\DeclareMathOperator{\dima}{dim_A}

\DeclareMathOperator{\vis}{Vis}

\DeclareMathOperator{\Tan}{Tan}

\DeclareMathOperator{\dist}{dist}
\DeclareMathOperator{\diam}{diam}
\DeclareMathOperator{\proj}{proj}

\DeclareMathOperator{\inter}{int}

\renewcommand{\atop}[2]{\genfrac{}{}{0pt}{}{#1}{#2}}

\newenvironment{labeledlist}[2][\unskip]
{ \renewcommand{\theenumi}{\textnormal{({#2}\arabic{enumi}{#1})}}
  \renewcommand{\labelenumi}{\textnormal{({#2}\arabic{enumi}{#1})}}
  \begin{enumerate} }
{ \end{enumerate} }

\begin{document}

\title{Visible part of dominated self-affine sets in the plane}

\author{Eino Rossi}
\address{[Eino Rossi]
        Department of Mathematics and Statistics, University of Helsinki \\
        P.O. Box 68  (Pietari Kalmin katu 5) \\
        00014 University of Helsinki\\
        Finland}
\email{eino.rossi@gmail.com}

\date{\today}
\thanks{The author was funded by the Academy of Finland through project Nos. 314829(Frontiers of singular integrals) and 309365(Quantitative rectifiability in Euclidean and non-Euclidean spaces).}

\begin{abstract}
  The dimension of the visible part of self-affine sets, that satisfy domination and a projection condition, is being studied. The main result is that the Assouad dimension of the visible part equals to 1 for all directions outside the set of limit directions of the cylinders of the self-affine set. The result holds regardless of the overlap of the cylinders. The sharpness of the result is also being discussed.
\end{abstract}

\maketitle


\section{Introduction}
For $e \in S^1$, let $\ell(e)$ denote the half line starting from origin and propagating to direction $e$. That is $\ell(e) =  \{ t e : t\geq 0 \}$. For a compact set $E\subset \R^2$ the visible part of $E$ in direction $e\in S^1$ is the set of points in $x\in \R^2$ that satisfy
\[
 ( \{x\} + \ell(e) ) \cap E = \{x\}.
\]
This set is denoted by $\vis^e E$. Let $\proj^e$ denote the orthogonal projection along the direction $e$. (Note that $\vis^e E$ may be different from $\vis^{-e} E$, but $\proj^e E = \proj^{-e} E$ always.) Consider the Hausdorff dimension of the visible part of a compact set $E$. If $\dimh E < 1$ then $\dimh \proj^e E = \dimh E$ for almost all $e\in S^1$ by Marstrand's projection theorem \cite{Marstrand1954}. Since $\vis^e E \subset E$ and $\proj^e \vis^e E = \proj^e E$, it follows that $\dimh \vis^e E = \dimh E$ for almost all $e\in S^1$. If $\dimh E \geq 1$, then still we have that $1 \leq \dimh \vis^e E$ for almost all $e\in S^1$, but the upper bound $\dimh \vis^e E \leq \dimh E$ should no longer be optimal for most $e\in S^1$. The visibility conjecture states that $\dimh \vis^e E = 1$ for almost all $e\in S^1$. Obviously one can not hope this to hold for all directions, since a graph of a function can have dimension greater than $1$ for example. Further, an example of Davies and Fast \cite{DaviesFast1978} shows that $\dimh \vis^e(K) = 2$ is possible for a dense $G_\delta$ set of directions. This is the furthest one can go, since recently Orponen \cite{Orponen2019vis} showed that it is impossible to have $\dimh \vis^e(K) = 2$ for set of directions of positive measure. It is rather easy to see that the visibility conjecture is false for the box counting dimension and thus for the Assouad dimension as well. This follows, since a countable set equals to its visible part for almost all directions and there exist compact countable sets with full box dimension. For example, one can simply consider $K = A \times A$, where $A=\{0\} \cup \{(S_n)^{-1}\}_{n=1}^\infty$ and $S_n = \sum_{k=1}^n 1/k$. For details, see Example \ref{ex:assouad_and_box}.

The visibility conjecture has been confirmed in a few special cases: J\"arvenp\"a\"a et.al. \cite{JarvenpaaEtAl2003} proved the conjecture for quasi-circles, Arhosalo et al. \cite{ArhosaloEtAl2012} confirmed that for fractal percolation the conjecture holds almost surely, and Falconer and Fraser \cite{FalconerFraser2013} showed that the conjecture holds for self-similar sets satisfying a projection condition and the open set condition so that the open set can be chosen to be convex. In all these cases, the authors actually verified the conjecture for the box dimension and for all directions $e\in S^1$. See also the recent work of J\"arvenp\"a\"a et.al. \cite{JarvenpaatSuomalaWu2021}.

One obvious variant of the problem is to consider the visible set from a given point instead of a direction. O'Neil \cite{ONeil2007} showed that for compact connected subsets of $\R^2$, the Hausdorff dimension of the visible part from a point $x\in \R^2$ is strictly less than the Hausdorff dimension of the original set, and it is uniformly bounded away from $2$, for almost all viewpoints $x$. An other related problem is to determine when $\vis^e E = E$. Orponen \cite{Orponen2014} showed that if $\dimh E > 1$, then the set of directions for which $\vis^e E = E$ has Hausdorff dimension at most $2-\dimh E$. On the other hand, it follows from the main result of \cite{RossiShmerkin2019}, that if $\dimb E < 1/2$, then $\vis^e E = E$ holds outside a set of directions of box dimension $2\dimb E$. For other related results, see for example \cite{Csornyei2000,BondLabaZahl2016,SimonSolomyak2006}.

In this paper I study the visible parts of self-affine sets. Domination and projection condition are standing assumptions throughout the paper. Theorem \ref{thm:general_dominated_vizibility} is the main result and it says that the Assouad dimension of the visible part equals to $1$ for all directions outside the set of the limit directions given by the affine dynamics. This theorem then has several corollaries. Corollary \ref{cor:carpet} says that for dominated self-affine carpets the Assouad dimension of the visible part equals to $1$ for all but two exceptional directions (that span the same line). Corollary \ref{cor:cone_visible} says that if the self-affine system satisfies the strong cone separation, then the Assouad dimension of the visible part equals to $1$ for almost all directions.  These results can be seen as rather strong, considering how easily the Assouad dimension jumps up in different situations. For example, it is well known that the fractal percolation has equal Hausdorff and box dimension $<2$ but full Assouad dimension, and Assouad dimension also tends to be maximal in projections in a way that is impossible for Hausdorff or box dimension \cite{Orponen2021proj}. Corollary \ref{cor:unique_visible} studies the case where the limit directions of the cylinders do not overlap too much, and states that the Hausdorff dimension of the visible part equals to $1$ for all directions in this case.

\begin{ack}
 I want to thank Bal\'{a}zs B\'{a}r\'{a}ny, Antti K\"aenm\"aki, and Tuomas Orponen for inspiring discussions on the topics of this paper. I also wish to thank the anonymous referee for the valuable comments on how to improve the quality of this paper.
\end{ack}

\section{Preliminaries and statement of the main result}
The purpose of this section is only to fix the setting of the paper and state the main result. In the next sections, along the course of the proof, I give more insight by explaining the geometry behind the assumptions and the result.

Throughout the paper, a direction means a unit vector $e\in S^1$ and orientation is an element of the projective space $\PPP^1$, that is, the metric space of lines in $\R^2$ that go through origin, and where the distance is measured by the angle between the lines. For a vector $e\in \R^2$, let $\left<  e \right> = \{ te : t\in\R \}$ denote the corresponding element of the projective space. It is sometimes more intuitive to think $S^1$ as a set of angles instead of unit vectors. This justifies the use of notations $\langle \theta \rangle := \{ t(\cos \theta,\sin\theta) : t\in\R \}$, $\proj^\theta := \proj^{(\cos \theta,\sin\theta)}$, $\vis^\theta := \vis^{(\cos \theta,\sin\theta)}$, and $\ell(\theta) := \ell((\cos\theta,\sin\theta))$ for $\theta \in\R$.

Let $A\colon \R^2\to \R^2$ be an invertible linear map, so that $A(B(0,1))$ is not a ball. Then $A(B(0,1))$ is an ellipse whose semiaxes have different lengths. Let $\alpha_1(A)$ be the length of the longer one of the semiaxes and let $\alpha_2(A)$ be the length of the shorter one. Equivalently $\alpha_k(A), k=1,2$ are the square roots of the eigenvalues of $A^T A$ (ordered so that the larger is $\alpha_1$). Also, set $\vartheta_1(A)\in \PPP^1$ to be the orientation of the longer semiaxes of $A(B(0,1))$. That is, $\vartheta_1(A) = \left< A \eta_1(A) \right>$, where $\eta_1(A)$ is the normalized eigenvector of $A^T A$ associated to the eigenvalue $\alpha_1(A)^2$. Likewise, set $\vartheta_2(A) = \left< A \eta_2(A) \right>$, where $\eta_2(A)$ is the normalized eigenvector corresponding to $\alpha_2(A)^2$. It is a basic fact that $\vartheta_1(A) \perp \vartheta_2(A)$ and $\eta_1(A) \perp \eta_2(A)$.

Let $\{A_i\}_{i=1}^\kappa$ be a collection of contractive invertible linear maps, let $\{c_i\}_{i=1}^\kappa$ be a collection of vectors in $\R^2$, and let $\fii_i (x)= A_i x + c_i$, for all $i\in \{1,\ldots, \kappa\}$. It standard that there exists a unique compact set $E$ satisfying
\[
 E = \bigcup_{i=1}^\kappa \fii_i(E).
\]
The set $E$ is called self-affine.

Set $\Sigma^* = \bigcup_{k\in\N} \{1,\ldots,\kappa\}^k$ and $\Sigma = \{1,\ldots,\kappa\}^\N$. Write $\Sigma^n$ for $\{1,\ldots,\kappa\}^n$ even though this is abusing the notation. Let $|\iii|$ denote the length of the word $\iii$. That is, $|\iii| = n$ whenever $\iii\in \Sigma^n$ and $|\iii|= \infty$, when $\iii\in\Sigma$. For $\iii,\jjj \in \Sigma$, let $\iii \wedge \jjj $ be the longest common beginning of $\iii$ and $\jjj$ and define a distance function $\roo$ in $\Sigma$ by setting $\roo(\iii,\jjj) = 2^{ -|\iii \wedge \jjj| }$, with the interpretation $2^{-\infty} = 0$. This makes $(\Sigma,\roo)$ a compact metric space.

For quantities $x$ and $y$, usually depending on $\iii$, the notation $x \lesssim y$ means that there is a constant $C>1$, that may depend on the self-affine set $E$, so that $x \leq C y$. Further, $x \approx y$ means that $x \lesssim y$ and $y \lesssim x$. 

For $\iii = (i_1,i_2,\ldots,i_n)\in \Sigma^*$, let $A_\iii = A_{i_1}A_{i_2}\dots A_{i_n}$ and for the sake of brevity, write $\alpha_k(\iii) = \alpha_k(A_\iii)$ and $\vartheta_k(\iii) = \vartheta_k(A_\iii)$ for $k=1,2$. Similarly, also write $\fii_\iii = \fii_{i_1}\circ \fii_{i_2} \circ \dots \circ \fii_{i_n}$. The line $\vartheta_1(\iii)$ is called the orientation of the cylinder $\fii_\iii(E)$, because the cylinder $\fii_\iii(E)$ is ``close'' to being a line segment that is parallel to the line $\vartheta_1(\iii)$, at least when $|\iii|$ is large. This phenomena is examined in more detail in Proposition \ref{prop:general_weak}. As usual, let $\pi\colon \Sigma \to E$ be the canonical projection defined by
\[
 \{\pi\iii\} = \bigcap_{n=1}^\infty \fii_{\iii|_n}(E).
\]
From time to time, the set $\fii_{\iii|_n}(E)$ is also denoted by $E_{\iii|_n}$. The system $\{A_i\}_{i=1}^\kappa$ is called dominated, or said to satisfy dominated splitting, if there are constants $\tau > 1$ and $n_0 \in \N$, so that $\alpha_1(\iii)  > \tau^{|\iii|} \alpha_2(\iii)$ for all $\iii \in \Sigma^*$, with $|\iii| \geq n_0$. Domination ensures the existence of the limit orientation for all symbols $\iii \in \Sigma$. The next lemma records this fact along with other useful properties of the limit orientations.

\begin{lemma}
 \label{lem:directions}
 Let $E$ be a dominated self-affine set. Then
 \begin{enumerate}
  \item\label{existence} $\vartheta_1(\iii) = \lim_{n \to \infty} \vartheta_1(\iii|_n)$ exists for all $\iii\in\Sigma$ and the convergence is uniform.
  \item\label{continuity} The map $\vartheta_1\colon \Sigma \to \PPP^1$ is uniformly continuous.
  \item\label{accumulation} $\vartheta_1(\Sigma)$ contains the accumulation points of the set $\{ \vartheta_1(\iii) : \iii\in\Sigma^* \}$.
  \item\label{invariance} $A_{\iii}\vartheta_1(\jjj) = \vartheta_1(\iii\jjj)$ for all $\iii\in\Sigma^*$ and $\jjj\in\Sigma$.
 \end{enumerate}
\end{lemma}
\begin{proof}
 The proof of \eqref{existence} is a direct modification of \cite[Lemma 2.1]{KaenmakiKoivusaloRossi2017}, where the existence of the limit in question is showed for almost all $\iii$. The proof there works for individual $\iii\in \Sigma$ for which
 \[
  \liminf_{n\to\infty}-\frac{1}{n} \log \frac{\alpha_2(\iii|_n)}{\alpha_1(\iii|_n)} > 0.
 \]
 In the setting of this paper, the domination implies the uniform bound $\log\tau > 0$ for the above liminf. The uniform bound also implies that the convergence is uniform. The part \eqref{continuity} follows from \eqref{existence}, and \eqref{accumulation} follows from \eqref{continuity} and compactness of $\Sigma$.
 
 To prove \eqref{invariance} it suffices to show that $A_\iii^{-1} \vartheta_1(\iii\jjj|_n)$ converges to $\vartheta_1(\jjj)$ as $n\to \infty$, since $A_{\iii}$ is a diffeomorphism. Write
 \[
  \eta_1(\iii\jjj|_n) = t_n \eta_1(\jjj|_n) + s_n \eta_2(\jjj|_n),
 \]
 and for now let $\theta_k(\jjj|_n)\in S^1$ be a unit vector with $\left< \theta_k(\jjj|_n) \right> = \vartheta_k(\jjj|_n)$ for $k=1,2$. Then it follows from domination that
 \begin{align*}
  A_\iii^{-1} \vartheta_1(\iii\jjj|_n)
  &=
  \left< A_\iii^{-1} A_\iii A_{\jjj|_n} \eta_1(\iii\jjj|_n) \right>
  =
  \left< t_n \alpha_1(\jjj_n)\theta_1(\jjj|_n) + s_n \alpha_2(\jjj_n)\theta_2(\jjj|_n) \right> \\
  &=
  \left< \theta_1(\jjj|_n) +  \frac{s_n\alpha_2(\jjj|_n)}{t_n\alpha_1(\jjj|_n)}\theta_2(\jjj|_n) \right>
  \to
  \vartheta_1(\jjj)
 \end{align*}
 as long as $t_n$ stays bounded away from zero. To show that it does, recall that $|A_{\iii\jjj|_n}\eta_1(\iii\jjj|_n)| = \max_{v\in S^1} |A_{\iii\jjj|_n} v|$. In particular,
 \[
  |A_{\iii\jjj|_n}\eta_1(\jjj|_n)| \leq |A_{\iii\jjj|_n}\eta_1(\iii\jjj|_n)|,
 \]
 where the left hand side is at least $\alpha_2(\iii) \alpha_1(\jjj|_n)$ and the right hand side is at most
 \[
  \alpha_1(\iii) |t A_{\jjj|_n}\eta_1(\jjj|_n) + s A_{\jjj|_n}\eta_2(\jjj|_n)| \leq
  \alpha_1(\iii) | t_n \alpha_1(\jjj|_n) + s_n\alpha_2(\jjj|_n) |.
 \]
 Thus the triangle inequality gives
 \[
  \frac{\alpha_2(\iii)}{\alpha_1(\iii)}
  \leq
  |t_n| + |s_n| \frac{\alpha_2(\jjj|_n)}{\alpha_1(\jjj|_n)}
 \]
 and so the domination implies that $|t_n| \geq 2^{-1} \alpha_2(\iii)  \alpha_1(\iii)^{-1}$ for large $n$.
\end{proof}
In addition to domination, a crucial assumption in this paper is the following projection condition.
\begin{definition}
 An affine IFS $\{\fii_i\}$ (or the invariant set $E$) satisfies the projection condition if $\PPP^1\setminus \{\vartheta_1(\jjj): \jjj\in\Sigma\} \neq \emptyset$ and if for all $e\in S^1$ with $\langle e \rangle\in \PPP^1\setminus \{\vartheta_1(\jjj): \jjj\in\Sigma\}$, there is $n_0$ so that $\proj^e \fii_\iii(E)$ is a non-trivial interval for all $\iii\in\Sigma^n$ and $n\geq n_0$.
\end{definition}
\begin{remark}
\label{proj_cond_rem}
 To check the projection condition in a specific case, it may be useful to note that affinie mappings preserve lines and convex hulls. (The convex hull of a set $E$ is the smallest convex set containing $E$.) That is, if $\ell$ is a line in $\R^2$ and $K$ is the convex hull of $E$ and $A$ is an invertible affine map, then $A\ell$ is also a line and $AK$ is the convex hull of $AE$. Asking if $\proj^e \fii_\iii(E)$ is an interval, is equivalent to asking if $\proj^{A_\iii^{-1}e}(E)$ is an iterval. Further, $\proj^{A_\iii^{-1}e}(E)$ is an interval if and only if every line $\ell$ of the form $\{y\} + \langle A_\iii^{-1} e \rangle$ that meets $K$ also meets $E$.
\end{remark}

As said, the purpose is to study the dimension of the visible part. I assume that the reader is familiar with basic notions of dimension. The Hausdorff dimension is denoted by $\dimh$, the box dimension by $\dimb$, and the Assouad dimension by $\dima$. The definitions of Hausdorff and box dimension one can find from almost any text book of fractal geometry or geometric measure theory (see for example \cite{Mattila1995}), and for Assouad dimension one can check for example \cite{Luukkainen1998}. If the reader is not interested in the Assouad dimension, then I just want to remark that the results are new also for the Hausdorff dimension and that the versions with Assouad dimension are just stronger since $\dimh K \leq \dima K$ for all sets $K$.

The main result is the following theorem.

\begin{theorem}
 \label{thm:general_dominated_vizibility}
 Let $E$ be a self-affine set satisfying the projection condition and the dominated splitting. Then $\dimh \vis^e(E) = \dima \vis^e(E) = 1$ for all $e \in S^1$ with $\langle e \rangle \not\in\{ \vartheta_1(\iii) : \iii\in \Sigma \}$.
\end{theorem}
The proof goes via weak tangents, defined as follows. Let $F_n$ be a sequence of compact sets in $\R^2$. Say that $F_n$ converges to a compact set $F\subset \R^2$ in $B(0,R)$, if
 \[
  \sup\{ \dist \left( F_n , x  \right) : x\in F \cap B(0,R) \} \to 0 \text{ as } n\to\infty
 \]
 and
 \[
  \sup\{ \dist \left(  x , F  \right) : x \in F_n \cap B(0,R) \} \to 0 \text{ as } n\to\infty.
 \]
For $x\in \R^2$ and $r>0$ we write $M_{x,r}$ for the magnification function that shifts $x$ to origin and scales with factor $r^{-1}$. That is
\[
 M_{x,r}(y) = \frac{y-x}{r}.
\]
Let $X\subset \R^2$ be compact. Then $W \subset B(0,1)$ is said to be a weak tangent of $X$, and written $W\in\Tan(X)$, if $M_{x_n,r_n}(X)$ converges to $W$ in $B(0,1)$ for some sequences $(x_n) \subset X$ and $r_n \searrow 0$. It is typical to consider the weak tangents as subsets of the unit ball (or the unit square), but this is just a convenient choice. One could as well consider the convergence in $B(0,R)$ for any fixed $R>0$ or for all $R>0$ to allow the weak tangents to be unbounded as well.

Assuming Proposition \ref{prop:general_weak_dimension}, which deals with the dimension of the weak tangents of the visible part, its proof of Theorem \ref{thm:general_dominated_vizibility} is rather simple.

\begin{proof}[Proof of theorem \ref{thm:general_dominated_vizibility}]
 Consider $e\in S^1$ with $\langle e \rangle \not\in \{ \vartheta_1(\iii) : \iii\in \Sigma \}$. By the projection condition, it holds that $\proj^e E_\iii$ is an interval for some $\iii \in \Sigma^*$, so $\dimh \vis^e E \geq 1$. Thus the task is to prove the upper bound. Proposition \ref{prop:general_weak_dimension} says that $\dimh W \leq 1$ for all weak tangents $W$ of $\overline{\vis^e E}$ (the closure is needed since the visible part is not necessarily closed). Recalling that the Assouad dimension of a compact set equals to the maximum of the Hausdorff dimensions of its weak tangents \cite[Proposition 5.8]{KOR}, it then follows that $\dimh \vis^e E \leq \dima \overline{\vis^e E} \leq \dimh W \leq 1$, were $W$ is a weak tangent of $\overline{\vis^e E}$ with maximal Hausdorff dimension.
\end{proof}

Considering the proof of Proposition \ref{prop:general_weak_dimension}, it is obvious that if $W$ is a weak tangent of $\overline{\vis^e E}$ with $M_{x_n,r_n}\overline{\vis^e E} \to W$ in $B(0,1)$, then (by passing to a subsequence if necessary) it also holds that $(M_{x_n,r_n} E) \cap B(0,1) $ converges to a weak tangent, say $T$, of $E$. Of course $W\subset T$, but unfortunately, it is not generally true that $W \subset \overline{\vis^e T}$ or $W \supset \overline{\vis^e T}$. In particular, one can not just take the weak tangent $T$ of $E$ of maximal dimension and expect it to have anything to do with the weak tangent of $\overline{\vis^e E}$ of maximal dimension. See example \ref{ex:sharp}. Instead, the strategy is to use the structure of the self-affine set and the weak tangents obtained in Section \ref{sec:weaktan} to show that $W \cap \overline{\vis^e T}$ can be covered by graphs of few well behaving functions and that $W \setminus \overline{\vis^e T}$ can be covered by a countable collection of lines. The arguments about visibility rely heavily on the projection condition.


The visibility conjecture asks if $\dimh \vis^\theta(E) = 1$ for almost all $e\in S^1$. So, the remaining step to confirm the conjecture in some special case, is to show that $\vartheta_1(\Sigma)$ is of measure zero (or technically, that the set $e\in\{ \langle e \rangle \in \vartheta_1(\Sigma) \}$ is of measure zero). Theorem \ref{cor:cone_visible} deals with this in the case where the self-affine system also satisfies the strong cone separation. In section \ref{sec:final}, I give an example where the visible part has large dimension in directions $\vartheta_1(\Sigma)$, showing that Theorem \ref{thm:general_dominated_vizibility} is sharp.

\section{Weak tangents of dominated self-affine sets}
\label{sec:weaktan}
This section deals with the structure of the weak tangent sets of self-affine sets satisfying the projection condition and dominated splitting. Recall that no separation conditions are required. The structure of tangents of self-affine sets under separation conditions has been studied in \cite{BandtKaenmaki2013,KaenmakiKoivusaloRossi2017,Mackay,KOR} for example.

%
To study the local structure of self-affine sets it is convenient to approximate the cylinders $\fii_\iii(E)$ by rectangles. The domination ensures that $\alpha_2(\iii) / \alpha_1(\iii) \to 0$ uniformly as $|\iii| \to \infty$. Therefore the approximation of $\fii_\iii(E)$ can be done with a ``very narrow'' rectangle if $|\iii|$ is large. This motivates the following definition.

\begin{definition}
\label{def:approxrectangle}
 For $\iii\in \Sigma^*$ define the approximating rectangle $R(x,r,\iii)$ to be the smallest closed rectangle that includes $M_{x,r}(E_\iii)$ and has sides parallel to $\vartheta_1(\iii)$ and $\vartheta_2(\iii)$. For any approximating rectangle $R$, the length of the sides parallel to $\vartheta_1(\iii)$ is denoted by $h(R)$ and the length of the sides parallel to $\vartheta_2(\iii)$ is denoted by $v(R)$. The orientation of cylinder $E_\iii$ is also called the orientation of the approximating rectangle $R(x,r,\iii)$.
\end{definition}

\begin{lemma}
 \label{lem:approx_rectengles}
 Let $E$ be a self-affine set satisfying the projection condition and the dominated splitting. Let $W$ be a weak tangent of $E$ with $M_{x_n,r_n}(E) \to W$ in $B(0,1)$ and let $x\in W$.
 
 Then there exists a sequence $\iii_n \in \Sigma^*$ of finite words and a sequence $R_n := R_n(x_n,r_n,\iii_n)$ of approximating rectangles so that $h_n := h(R_n) \to \infty$ and $v_n := v(R_n) \to 0$ and $\dist(R_n,x)\to 0$.
\end{lemma}
\begin{proof}
 Since $x \in W$, there exists sequences $(z_n) \subset B(0,1)$ and $(\iii_n) \subset \Sigma^*$ so that $z_n \in M_{x_n,r_n}(E_{\iii_n})$ and $z_n \to x$. Furthermore, $\iii_n$ can be chosen so that $\alpha_2(\iii_n) \approx r_n/n$. By setting $R_n = R(x_n,r_n,\iii_n)$ it is obvious that $\dist(R_n,x)\to 0$. Note also that $\alpha_1(\iii_n) \approx r_n h_n$ and $\alpha_2(\iii_n) \approx r_n v_n$. By domination, there exist $\tau > 1$, so that $\alpha_1(\iii_n) \geq \tau^n \alpha_2(\iii_n)$ for large $n$. Thus,
 \[
  v_n \approx \alpha_2(\iii_n) / r_n \approx 1/n \to 0
 \]
 and
 \[
  h_n \approx \alpha_1(\iii_n) / r_n  \geq \tau^{n} \alpha_2(\iii_n) / r_n \approx \tau^{n} /n \to \infty.
 \]
\end{proof}
After the previous lemma, it is intuitive that the weak tangent contains lines and half lines pointing in different directions. Due to the obvious connection, it is natural to call such sets Kakeya type sets.

\begin{definition}
 Let $X \subset \R^2$ and fix $\theta_x \in S^1$ for all $x\in X$. A set of the form
 \[
  \bigcup_{x\in X} \{x\} + \ell(\theta_x)
 \]
 is called a Kakeya type set. The collection $\{ \theta_x \}_{x\in X}$ is called the direction set of the Kakeya type set.
\end{definition}

\begin{proposition}
 \label{prop:general_weak}
 Let $E$ be a self-affine set satisfying the projection condition and the dominated splitting, and let $W$ be a weak tangent of $E$. Then $W = D \cap B(0,1)$, where $D$ is a Kakeya type set with direction set $\Lambda$, that satisfies $\langle \theta \rangle \in \vartheta_1(\Sigma)$ for all $\theta \in \Lambda$.
\end{proposition}
\begin{proof}
 Let $M_{x_n,r_n}E$ converge to $W$ in $B(0,1)$ and fix $x\in W$. By Lemma \ref{lem:approx_rectengles}, there is a sequence $R_n = R(x_n,r_n,\iii_n) $ of approximating rectangles with $h_n \to \infty$ and $v_n \to 0$ so that $\dist(x,R_n) \to 0$. Recall that $R_n$ has orientation $\vartheta_1(\iii_n)$. Since $h_n \to \infty$, at least one of the shorter sides of $R_n$ is outside $B(0,1)$. One can now fix a direction that points to this side that is far away. To put this precise, choose a short side of $R_n$ that does not meet $B(0,1)$ and extend this line segment to an infinite line from both ends and call this line $\xi$ for now (if there are two choices for the short side, then it does not matter which one is chosen). Note also that $\xi$ is perpendicular to $\vartheta_1(\iii_n)$. Then choose $\theta_n$ so that $ \langle \theta_n \rangle = \vartheta_1(\iii_n)$ and $\{x\} + t\theta_n$ meets $\xi$ for some $t>0$. By passing to a sub-sequence, one can also assume that $\theta_n$ converges to some $\theta_x \in S^1$ and by Lemma \ref{lem:directions} part \eqref{accumulation} it holds that $\langle \theta_x \rangle \in \vartheta_1(\Sigma)$. By the projection condition there is at least one $e\in S^{1}\setminus \vartheta_1(\Sigma)$ so that $\proj^e \fii_\iii(E)$ is an interval whenever $|\iii|$ is large. Further by compactness of $\vartheta_1(\Sigma)$, the approximating rectangles $R_n$, with large $n$, have orientation bounded away from $\langle e \rangle$. Thus, by the projection condition and the choices made above, it is clear that
 \[
  \left( \{x\} + \ell(\theta_x) \right) \cap B(0,1) \subset W.
 \]
 Trivially $W \subset \bigcup_{x\in W} \{x\}$. So, by taking union over all $x\in W$, it then follows that
 \begin{equation}
  \label{eq:union}
  W = \bigcup_{x\in W} \left( \{x\} + \ell(\theta_y) \right) \cap B(0,1),
 \end{equation}
 which is exactly what was claimed.
\end{proof}

\begin{remark}
 For sure, the union in \eqref{eq:union} is not optimal, meaning that it is not necessary to take the union over all $x\in W$. In particular, if $x\in W$ then also $z_t = x + t\ell(\theta_x) \in W$ for all small $t>0$ at least. If $\ell(\theta_{z_t}) = \ell(\theta_x)$ then the union doesn't need to be over $z_t$ at all. Note however that even tough $z_t$ is on a line $\{x\} + \ell(\theta_x)$ it may be that $\ell(\theta_{z_t}) \neq \ell(\theta_x)$ due to overlap of cylinders in the original self-affine set $E$.
 %
\end{remark}

\section{Proofs of the main results}
In this section I finish the proof of Theorem \ref{thm:general_dominated_vizibility}. As discussed earlier, all that is left to do is to prove Proposition \ref{prop:general_weak_dimension}. After this it is time to focus on the Corollaries \ref{cor:carpet}, \ref{cor:cone_visible}, and \ref{cor:unique_visible}, that deal with the size of the exceptional set of directions, verifying the visibility conjecture in different special cases.

\begin{proposition}
 \label{prop:general_weak_dimension}
 Let $E$ be a self-affine set satisfying the projection condition and the dominated splitting and let $ e \in S^1 $ so that $ \langle e \rangle \not\in\vartheta_1(\Sigma)$. Then $\dimh W \leq 1$ for all $W \in \Tan( \overline{\vis^e(E)} )$.
\end{proposition}

As mentioned earlier the strategy is to cover $W$ with graphs of nice functions and a collection of vertical lines. With this in mind, recall some basic facts. For $f\colon\R \to \R$, let $\GG(f)$ denote the graph of $f$. That is, $\GG(f) = \{ (x,y)\in\R^2 : f(x) = y \}$. A function $f\colon \R \to \R$ satisfying $f(t) - f(s) \leq L (t-s)$ for some $L>0$ and for all $t \geq s$ is called semi-decreasing. Also, $f$ is said to be semi-increasing if $-f$ is semi-decreasing and $f$ is called semi-monotone if it is semi-decreasing or semi-increasing. The aim is to use graphs of semi-monotone functions for the coverings, so the first thing to do is to check that their graphs are nice enough.
\begin{lemma}
 \label{lem:semimonotone}
 Let $f\colon \R \to \R$ be semi-monotone. Then $\dimh \GG(f) = 1$. Further the set of discontinuity points of $f$ is at most countable.
\end{lemma}
\begin{proof}
 It is standard that the claim holds for monotone functions. By symmetry, it is enough to show the semi-decreasing case. So, assume that $f$ is semi-decreasing and that the involved constant is $L$. Define $\fii\colon \R \to \R$ by $\fii(t) = L \cdot t$ and consider $g = f - \fii$. Since $g$ is monotone and $\fii$ is Lipschitz, the second claim follows. Also, $\dimh \GG(g) = 1$ since $g$ is monotone. On the other hand, $\GG(g) = \Psi ( \GG(f) ) $, where $\Psi\colon\R^2 \to \R^2$ is defined by $\Psi(x,y) = (x , y - \fii(x) )$. Clearly $1 \leq \dimh \GG(f) $, since $\proj^{- \frac{\pi}{2} } \GG(f) = \R$. On the other hand, it is easy to see that $\Psi$ is bi-Lipschitz, so $1 \leq \dimh \GG(f)  = \dimh \Psi ( \GG(f) ) = \dimh \GG(g) = 1$
\end{proof}

\begin{proof}[Proof of Proposition \ref{prop:general_weak_dimension}]
 Fix a direction $e$ as in the claim and let $x_n,r_n$ be sequences so that $M_{x_n,r_n} (\overline{ \vis^e (E) }) \to W$ in $B(0,1)$. After passing to subsequence if necessary, it can also be assumed that $M_{x_n,r_n} (E)$ converges to some weak tangent $T$ in $B(0,1)$. By Proposition \ref{prop:general_weak}, the weak tangent $T$ is a Kakeya type set, so let $X\subset T$ and $\{\theta_x\}_{x\in X}\subset S^1$, so that $T = \bigcup_{x\in X} (\{x\} + \ell(\theta_x))$. By assumption $\pm e \neq \theta_x$ for all $x\in X$. Without loss of generality, assume that $e = (0,-1)$. Let $\beta = \min\{ | \sphericalangle ( \theta_x , \pm\pi/2 ) |\}$ and $\theta = \pi/2 - \beta$. Compactness of $\vartheta_1(\Sigma)$ ensures that $\beta$ is strictly positive. Still, without loss of generality, assume that $\dimh W = \dimh W \cap B(0, 2^{-1}\cos\theta)$, so it suffices to estimate the dimension of $W' = W \cap B(0, 2^{-1}\cos\theta)$. Set $\gamma := 2^{-1}\cos\theta$. The reason of focusing on this smaller ball inside $B(0,1)$, is merely a technicality and there is no need for the reader to worry about this too much. In a nutshell, if $\ell$ is a line or half line that meets $B(0,1)$, then $(\proj^{ e } \ell) \cap [-1,1]$ may be different from $\proj^{ e } (\ell \cap B(0,1))$. The choice of $\gamma$ ensures that if $\ell$ is a line or half line included in the weak tangent, and it meets $B(0,\gamma)$, then $(\proj^{ e } \ell) \cap [-\gamma,\gamma]$ equals to $[-\gamma,\gamma] \cap \proj^{ e } (\ell \cap B(0,1))$.
 
 First divide $T$ into three sets that each have nice enough geometry. Recall that $T$ consists of line segments, and only the lines that hit $B(0,\gamma)$ are meaningful. If $L$ is a collection of lines and half lines so that $\left(\bigcup_{\ell\in L} \ell\right) \cap B(0,1) = T$, and $L'\subset L$ consists of those elements that meet $B(0,\gamma)$, then set
 \begin{align*}
  L_T &= \{ \ell \in L' : \sharp( \ell \cap \partial B(0,1) ) = 2 \} \\ 
  L_R &= \{ \ell \in L' : \ell = \{ x \} + \ell(\theta_x), \text{ with } \cos\theta_x > 0, \text{ and } |x| < 1 \} \\ 
  L_L &= \{ \ell \in L' : \ell = \{ x \} + \ell(\theta_x), \text{ with } \cos\theta_x < 0, \text{ and } |x| < 1  \}
 \end{align*}
 For the lines in $L_T$ there are two possibilities. According to Lemma \ref{lem:approx_rectengles}, for $\ell\in L_T$, it may be that there exists a sequence of approximating rectangles $R_n$ converging to $\ell$ in $B(0,1)$. In this case set $\ell \in L_{TT}$. If this is not the case, then, since every approximating rectangle can have at most one short side meeting $B(0,1)$, there are two sequences of approximating rectangles, say $R_n$ and $S_n$, so that $R_n \cup S_n \to \ell$ in $B(0,1)$. (There might be many more that could be chosen, but it is enough to consider these two.) Since a small neighborhood of the vertical orientation is excluded, it makes sense to talk about left and right sides of these rectangles, referring to the shorter sides that are most right and most left. Assume that the left side of $R_n$ does not meet $B(0,1)$ and the right side of $S_n$ does not meet $B(0,1)$. By passing to subsequences, one can assume that there are $x,z\in\ell$ so that $R_n$ converges to $x + \ell(\theta_x)$ in $B(0,1)$, with $\cos\theta_x < 0$ and that $S_n$ converges to $z + \ell(\theta_z)$ in $B(0,1)$, with $\cos\theta_z > 0$. In this case, set $x + \ell(\theta_x) \in T_{TL}$ and $z + \ell(\theta_z) \in T_{TR}$.
 
 Finally, set
 \begin{equation*}
  T_T = \overline{ \bigcup_{\ell \in L_{TT}} \ell }, \qquad
  T_R = \overline{ \bigcup_{\ell \in L_R \cup L_{TR}} \ell }, \qquad
  T_L = \overline{ \bigcup_{\ell \in L_L \cup L_{TL}} \ell },
 \end{equation*}
 and $T' = T_T \cup T_L \cup T_R$. The closures are taken to ensure that $T_i\cap B(0,1)$ is a compact set for all $i\in\{T,R,L\}$. Obviously $T \cap B(0,\gamma) = T'\cap B(0,\gamma)$. Note that despite the closures, the sets $T_i$ are still collections of lines, since the closure of any set of lines in $\R^2$ is still a set of lines in $\R^2$. Most importantly, for each line in $T_i$, there still exists a sequence $R_n$ of approximating rectangles with $R_n\to T_i$, since they existed for all lines of $L_i$.
 
 Now it is time to estimate $\dimh (\vis^e T' )$. Trivially, $\vis^e T' \subset \vis^e T_T \cup \vis^e T_L \cup \vis^e T'_R$ so it suffices to consider $ \vis^e T_i$ for $i\in \{T,L,R\}$ separately. Of course, some of the sets $T_i$ may be empty, but at least one of them is nonempty since $T\cap B(0,\gamma)$ is nonempty.

 Start with $T_R$. Note that $[-\gamma,\gamma] \cap \proj^{e} T_R = [u,\gamma] =: I_R$ for some $u$. Consider the function $f_R \colon I_R \to \R$ defined by $f(x) = \min\{ y : (x,y) \in T_R \}$. Let $\Gamma$ denote the strip $[-\gamma,\gamma] \times \R$. Obviously $\GG(f_R) = \vis^e (T_R \cap \Gamma)$. Consider $s,t\in I_R$, with $s < t$. Since $(s,f(s))$ is on a line segment $\ell \in L_R$, with $ t \in [s,\gamma] \subset \proj^e \ell$, it is clear that $f(t) - f(s) \leq  \tan\theta (t-s)$, so $f_R$ is semi-decreasing, and $\dimh G(f_R) \leq 1$ by Lemma \ref{lem:semimonotone}.
 
 Similarly, define $f_L \colon I_L \to \R$ by setting $f(x) = \min\{ y : (x,y) \in T_L \}$. Again, $\GG(f_L) = \vis^e (T_L \cap \Gamma)$. Consider $s,t\in I_L$, with $s < t$. Since $(t,f(t))$ is on a line segment $\ell \in L_L$, with $ s \in [-\gamma,t] \subset \proj^e \ell$, it is clear that $f(t) - f(s) \geq - \tan\theta(t-s)$. Thus $f_L$ is semi-increasing, and $\dimh G(f_L) \leq 1$ by Lemma \ref{lem:semimonotone}.
 
 Finally, define $f_T \colon [-\gamma,\gamma] \to \R$ by $f(x) = \min\{ y : (x,y) \in T_T \}$ and note that $f_T$ is Lipschitz. All in all, the above considerations show that $\dimh W' \cap \vis^e T' \leq 1$.
 
 Then it is time to estimate $\dimh (W' \setminus \vis^e T')$. The aim is to show that $ W' \setminus \vis^e T'$ can be covered by a countable collection of vertical line segments. More specifically, by two collections of vertical lines that are parametrized by a) the discontinuity pints of $f_i$, b) the boundary points of $I_i$. Considering the first case, Lemma \ref{lem:semimonotone} showed that a semi-monotone function can have only countably many points of discontinuity. Hence, set
 \[
  L_D := \bigcup_{i=L,R,T} \bigcup \{ \{s\} \times \R : f_i \text{ is discontinuous at } s\}.
 \]
 For the second case, let $\{t_i\}_{i=1}^6$ be the endpoints of the intervals $I_k, k = T,R,L$, and set
 \[
  L_B := \bigcup_{i} \{t_i\} \times \R
 \]
 The final step is to show that $W' \setminus \vis^e T' \subset L_D \cup L_B$. If this is not the case, then there exists a point $\omega = (\omega_1,\omega_2)\in W' \setminus (\vis^e T' \cup L_D \cup L_B)$. Since $\proj^{ e } W' \subset I_T \cup I_R \cup I_L$ and $\omega$ is not visible, there exists $x = (\omega_1,x_2) \in \vis^e T' \cap \ell$ with $x_2 < w_2$ and $\ell \in T_i$ for some $i=T,R,L$.
 
 Assume first that $i = T$. Then, by the choices made above, there is a sequence $R_n$ of approximating rectangles, with side lengths $h_n \to \infty$ and $v_n \to 0$, converging to $\ell$ in $B(0,1)$. Let $\omega^n$ be a sequence so that $\omega^n \to \omega$ and $\omega^n \in M_{x_n,r_n} (\vis^e E)$. (Recall that $M_{x_n,r_n} (\overline{\vis^e E}) \cap B(0,1) \to W$.) Let $n$ be so large that $\dist (R_n,x) < (\omega^n_2 - x_2) / 2$ and $v_n / \cos \theta < (\omega^n_2 - x_2) / 2$ for all large $n$. This is possible since in both inequalities the left hand side converges to zero and $(\omega^n_2 - x_2)$ converges to $(\omega_2 - x_2) > 0$. Now, by the projection condition, there exists a point $z_n \in R_n \cap M_{x_n,r_n} (E)$ so that $z_n \in \{\omega^n\} + \ell(-\pi/2)$ implying that $\omega^n \not\in M_{x_n,r_n} (\vis^e E)$, which is a contradiction.
 
 Then assume that $i = R$. Assume also that $\omega_1 \in \inter(I_R)$ and that $f_R$ is continuous at $\omega_1$, since otherwise $\omega$ is covered by $L_D$ or $L_B$. Thus there exists $z = (z_1,z_2)\in \vis^e T_R$ with $z_1<w_1$ and $z_2 = f_R(z_1)$ and $z_2 + |z_1 - x_1| \tan\theta < (\omega_2 -x_2) /4$. Again, there is a sequence $R_n$ of approximating rectangles and points $z^n \in R_n$ with $z^n\to z$. Let $\omega^n$ be a sequence so that $w^n \to \omega$ and $\omega^n \in M_{x_n,r_n} (\overline{\vis^e E})$. When $n$ is so large that
 \begin{equation}
 \label{eq:displayformulas}
  \begin{split}
  &z_2^n + |z^n_1 - x_1| \tan\theta < x_2 + (\omega_2 -x_2) \frac{3}{8}, \\
  &v_n \cos(\theta)^{-1} < |\omega_2-x_2|/8, \\
  &\omega^n_2 - |\omega^n_1 - \omega_1| \tan\theta  >  \omega_2 - |\omega_2 - x_2| /4,
  \end{split}
 \end{equation}
 the projection condition implies that $\{\omega^n\} + \ell(-\pi/2) \cap R_n \neq \emptyset$ for all large $n$ implying that $\omega\not\in W'$. See Figure \ref{fig:explaindisplayformulas} for clarification. The case $i = L$ is symmetric to the case $i = R$.
\begin{figure}
\centering
\begin{tikzpicture}[scale=1.5]
 \draw[thick] (0,0) -- (0,4);
 \foreach \n in {1,...,3}
 {
 \draw[thick] (-0.1,\n)--(0.1,\n);
 }
 \coordinate (w) at (0,4);
 \filldraw (w) circle (0.7pt);
 \node[right] at (w) {$\omega$};
 %
 \coordinate (ww) at (0.3,3.5);
 \filldraw (ww) circle (0.7pt);
 \node[right] at (ww) {$\omega^n$};
 %
 \coordinate (x) at (0,0);
 \filldraw (x) circle (0.7pt);
 \node[right] at (x) {$x$};
 %
 \coordinate (z) at (-.4,.2);
 \filldraw (z) circle (0.7pt);
 \node[right] at (z) {$z$};
 %
 \coordinate (zz) at (-0.8,0.4);
 %
 \draw[thick, dotted] (z)--(0,1);
 %
 \begin{scope}[rotate=60 , shift={(zz)} ]
 \draw[white, top color=white, bottom color=black!50] (-0.1,-0.1) rectangle (4,0.1);
 \draw[white,thick] (-0.1,-0.1) rectangle (4,0.1);
 \node[] at (4,0) {$R_n$};
 \end{scope}
 \filldraw (zz) circle (0.7pt);
 \node[left] at (zz) {$z^n$};
 %
 %
 \draw[thick,->] (ww) -- ++ (0,-2);
\end{tikzpicture}
\caption{This picture explains the formulas \eqref{eq:displayformulas}. When the approximating rectangle $R_n$ is narrow, and $\omega^n$ is near $\omega$, and $z^n$ is near $z$, the projection condition ensures that $\omega^n$ is not visible.}
\label{fig:explaindisplayformulas}
\end{figure}
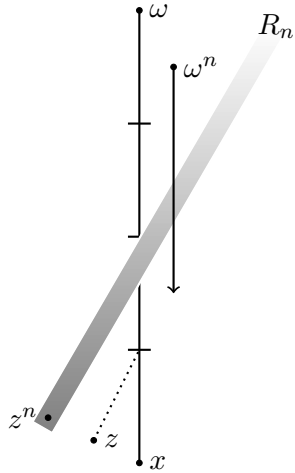
 
 The conclusion now is that $W' \subset \GG(f_T) \cup \GG(f_R) \cup \GG(f_L) \cup L_D \cup L_B$ and each element in the union has Hausdorff dimension $1$, so the proof is finished.
\end{proof}

Considering the visibility conjecture, there is still the question whether $\HH^1(\vartheta_1(\Sigma)) = 0$. For dominated self-affine carpets, this is true. A self-affine set is called a carpet if all $A_i$ are diagonal matrixes. If a carpet is dominated, then $\vartheta_1(\Sigma)$ is a singleton - it is either the horizontal or the vertical orientation. Thus theorem \ref{thm:general_dominated_vizibility} immediately implies that the visibility conjecture holds in this class.
\begin{corollary}
 \label{cor:carpet}
 Let $E$ be a self-affine carpet satisfying the projection condition and the dominated splitting. Then $\dimh \vis^e(E) = \dima \vis^e(E) = 1$ holds for all except possibly one $e \in S^1$ and its opposite $-e$.
\end{corollary}

To verify the visibility conjecture in a more general setting, consider the self-affine sets satisfying the ``strong cone separation'' introduced in \cite{KaenmakiShmerkin2009}: assume that there is a cone $X \subset \R^2$ so that $A_i(X) \subset \inter(X)$ and $A_i^T(X) \subset \inter(X)$, and for all $i$ and that
\begin{equation}
 \label{eq:cone_separation}
 A_i(X) \cap A_j(X) = \emptyset 
\end{equation}
for all $i\neq j$. As is intuitive, a cone is a union of set of lines through origin in $\R^2$ that have bounded angle from some fixed line. In what follows, the cone $X$ is understood as a subsets of $\R^2$ or $\PPP^1$ depending on the situation, and this should not cause any confusion. So, equivalently, a cone is an interval in the projective space $\PPP^1$. As discussed in the proof of \cite[Lemma 4.1]{KaenmakiShmerkin2009} it follows that $\eta_1(A_\iii) , \vartheta_1(A_\iii) \in X$ and $\eta_2(A_\iii) , \vartheta_2(A_\iii) \not\in X$ for all $\iii\in\Sigma^*$. Further, without loss of generality, assume that $\eta_1(\iii)$ is uniformly separated from $X^c$ for all $\iii$ independently of the length $|\iii|$. This follows simply by choosing $X'$ to be the minimal cone that includes $\cup_i ( A_i(X) \cup A_i^T(X) )$ and then applying the previous deduction to the cones $A_i(X')$. (Note that the strong cone separation holds with the cone $X'$ as well.)
\begin{corollary}
 \label{cor:cone_visible}
 Let $E$ be a self-affine satisfying the projection condition and the strong cone separation. Then $\dimh \vis^e(E) = \dima \vis^e(E) = 1$ for almost all $e\in S^1$.
\end{corollary}

\begin{proof}
 The strong cone separation implies domination \cite[Theorem B]{BochiGourmelon2009}, so by Theorem \ref{thm:general_dominated_vizibility} it is enough to show that $\HH^1(\vartheta_1(\Sigma)) = 0$. This would certainly follow form $\dima \vartheta_1(\Sigma) < 1$ and this in turn follows if $\vartheta_1(\Sigma)$ is porous. Recall that a subset $Y$ of some metric space is porous if there are constants $r_0,\alpha>0$ so that for all $y\in Y$ and $r<r_0$, there is $x\in B(y,(1-\alpha)r)$ so that $B(x,\alpha r) \cap Y = \emptyset $. For the connection between Assouad dimension and porosity, see for example \cite{Luukkainen1998}.
 
 
 Let $X$ be the cone from \eqref{eq:cone_separation}. Since $X$ is a union of lines through origin, it can also be considered as a subset of the projective space. Recall that $\PPP^1$ is a metric space where the distance is measured by $\sphericalangle(\cdot,\cdot)$, the angle between the corresponding lines in $\R^2$ (which lies in the interval $[0,\pi/2]$ as usual). An invertible linear mapping $A\colon \R^2 \to \R^2$ can naturally be interpreted as a mapping on the projective space as well since $A(\langle v \rangle) = \langle A(v) \rangle$ by linearity. Thus the mappings $\{ A_i \}_{i=1}^\kappa$ naturally generate collections of nested compact sets inside $X$, since $A_i(X)\subset X$ for all $i$. To be more precise, I claim that setting $E_\iii =  A_\iii(X)$, the collection $\{E_\iii\}_{\iii\in\Sigma^*}$ satisfies the conditions
 \begin{labeledlist}{M}
  \item\label{Moran1} $E_{\iii_n} \subset E_{\iii|_{n-1}}$ for all $\iii\in \Sigma$ and $n\in\N$
  \item\label{Moran2} $\diam (E_{\iii|_n} ) \to 0 $ as $n\to \infty$ for all $\iii\in \Sigma$.
 \end{labeledlist}
 Part \ref{Moran1} is obviously true, but verifying part \ref{Moran2} requires some work. If these conditions hold, then the collection $\{E_\iii\}_{\iii\in\Sigma^*}$ is called a Moran construction and there is a unique compact set
 \[
  Y = \bigcup_{\iii\in\Sigma} \bigcap_{n\in\N} E_{\iii|_n}.
 \]
 Further knowledge about Moran constructions is not needed, but the interested reader can check for example\cite{KaenmakiRossi2016}. For this particular example, the key point is that by combining \ref{Moran1}, \ref{Moran2}, with Lemma \ref{lem:directions} parts \eqref{accumulation} and \eqref{invariance}, and the fact that $\eta_1(A_\iii)\in X$ for all $\iii\in\Sigma$, it follows that $Y$ equals to $\vartheta_1(\Sigma)$.
 
 Now, to see that also $\ref{Moran2}$ holds, fix $\iii\in \Sigma^n$, and $a,b\in X$. Further, fix unit vectors $t\eta_1(\iii) + s\eta_2(\iii)$ and $u\eta_1(\iii) + v\eta_2(\iii)$ so that
 \[
  a = \langle t\eta_1(\iii) + s\eta_2(\iii) \rangle\quad\text{and}\quad b = \langle u\eta_1(\iii) + v\eta_2(\iii) \rangle.
 \]
  It immediately follows that
 \begin{align*}
  A_\iii(a) &= \langle A_\iii( t\eta_1(\iii) + s\eta_2(\iii) ) \rangle = \langle t A_\iii\eta_1(\iii) + s A_\iii\eta_2(\iii) \rangle \\
  A_\iii(b) &= \langle A_\iii( u\eta_1(\iii) + v\eta_2(\iii) ) \rangle = \langle  u A_\iii\eta_1(\iii) + v A_\iii\eta_2(\iii) \rangle
 \end{align*}
 Further, it is no restriction to assume that $t$ and $u$ are non-negative. Since $\langle \eta_2(\iii) \rangle \not\in X$, there exists $\delta > 0$, so that $t,u > \delta$ and $|s|,|v| < 1-\delta$. Moreover, $\delta$ can be chosen to be independent of $a,b,\iii$ and the level $n$, since $\langle \eta_2(\iii) \rangle$ is uniformly separated from $X$. Since $\delta > 0$ is fixed, there exists $M > 1$ so that
 \begin{equation}
  \label{eq:tan_bi_Lip}
  |\gamma - \beta| \leq |\tan\gamma - \tan\beta| \leq M|\gamma - \beta|
 \end{equation}
for angles $\gamma,\beta \in [-\pi/2 + \delta/2, \pi/2 - \delta/2]$ and $M \delta \geq \pi - \delta$. Therefore, it follows that
 \begin{equation}
  \label{eq:ab_contracting}
  \sphericalangle( A_\iii(a) ,  A_\iii(b) )
  \leq
  \frac{ \|A_\iii \eta_2(\iii) \| }{ \| A_\iii \eta_1(\iii) \| } \left| \frac{ s }{ t } - \frac{ v }{ u } \right|
  =
  \frac{\alpha_2(\iii)}{\alpha_1(\iii)} \left| \frac{ s }{ t } - \frac{ v }{ u } \right|
  \leq
  \frac{\alpha_2(\iii)}{\alpha_1(\iii)} M \pi
  \leq
  \tau^{-n} M \pi,
 \end{equation}
 and $\tau^{-n}\to 0$ as $n\to \infty$, which proves \ref{Moran2}.
 
 It now suffices to show that $\{E_\iii\}_{\iii\in \Sigma^*}$, satisfies the following bounded distortion property: there are constants $k_0\in\N$ and $D > 1$ so that
 \[
  \frac{ d^*(\iii) }{ d_*(\iii) }
  \leq
  D
 \]
 for all $\iii\in\Sigma^k$ and $k \geq k_0$, where
 \[
 d^*(\iii) = \sup_{\atop{a,b\in X}{a\neq b}} \frac{\sphericalangle( A_\iii(a) , A_\iii(b)) }{\sphericalangle(a,b)}
 \qquad\text{ and }\qquad
 d_*(\iii) = \inf_{\atop{a,b\in X}{a\neq b}} \frac{\sphericalangle( A_\iii(a) , A_\iii(b)) }{\sphericalangle(a,b)}
 \]
 If the bounded distortion holds, then let $I$ be a gap between two neighboring first level cylinders $A_i(X)$ and $ A_j(X)$. Note that $I$ exists due to the strong cone separation. Let $r>0$ and $\theta \in Y$. Let $\iii$ be a finite word with $A_{\iii}(Y) \subset B(\theta,r)$ but $| A_{\iii}(Y)|\gtrsim r$. Then, due to the separation of the cones, there is a gap $G := A_{\iii}(I) \subset B(\theta,r)$ and
 \begin{align*}
  \frac{|G|}{r}
  &\gtrsim
  \frac{ d_*(\iii) |I| }{ | A_\iii(Y)| }
  \geq
  \frac{ d_*(\iii) |I| }{d^*(\iii) |Y| }
  \geq
  D^{-1}\frac{|I|}{|Y|}
 \end{align*}
 which shows that $Y$ is porous.
 
 To show the bounded distortion property, fix $\iii\in \Sigma^k$, $a,b\in Y$. There is a small annoying technicality that the angle between lines in $X$ may be realized ``outside'' $X$ if the opening angle of $X$ is larger than $\pi/2$. Therefore, let $k_0$ be so large that $\sphericalangle( A_\iii(x) ,  A_\iii(y) ) \leq \delta$ for all lines $x,y\in X$ when $|\iii|\geq k_0$ and assume $k\geq k_0$. 
 
 As before, fix unit vectors $t\eta_1(\iii) + s\eta_2(\iii)$ and $u\eta_1(\iii) + v\eta_2(\iii)$ so that
 \[
  a = \langle t\eta_1(\iii) + s\eta_2(\iii) \rangle\quad\text{and}\quad b = \langle u\eta_1(\iii) + v\eta_2(\iii) \rangle.
 \]
  where $t$ and $u$ are positive. Next, note that $ \frac{ s }{ t } = \pm\tan \sphericalangle( \langle \eta_1(\iii) \rangle , a )$ depending on if $s$ is positive or negative, and that a similar formula holds for $\frac{ v }{ u }$. Combining this with \eqref{eq:tan_bi_Lip} gives
 \begin{equation}
  \label{eq:ab_upper1}
  \sphericalangle( A_\iii(a) ,  A_\iii(b) )
  \leq
  \frac{ \|A_\iii \eta_2(\iii) \| }{ \| A_\iii \eta_1(\iii) \| } \left| \frac{ s }{ t } - \frac{ v }{ u } \right|
  =
  \frac{\alpha_2(\iii)}{\alpha_1(\iii)} \left| \frac{ s }{ t } - \frac{ v }{ u } \right|
  \leq
  \frac{\alpha_2(\iii)}{\alpha_1(\iii)} M \sphericalangle( a , b ).
 \end{equation}
 for lines $a,b\in Y$ with $\sphericalangle( a , b ) \leq \delta$. If $\sphericalangle( a , b ) > \delta$, then it could be that the angle is realized outside $X$, and the last inequality in the above estimate may not hold. In this case recalling the choice of $M$ still gives
  \begin{equation}
  \label{eq:ab_upper2}
  \sphericalangle( A_\iii(a) ,  A_\iii(b) )
  \leq
  \frac{\alpha_2(\iii)}{\alpha_1(\iii)} \left| \frac{ s }{ t } - \frac{ v }{ u } \right|
  \leq
  \frac{\alpha_2(\iii)}{\alpha_1(\iii)} M (\pi - \delta)
  \leq
  \frac{\alpha_2(\iii)}{\alpha_1(\iii)} M^2  \sphericalangle( a , b ).
 \end{equation}
 Since $\sphericalangle( A_\iii(a) , A_\iii(b) ) \leq \delta$ by the choice of $k_0$, the lower estimate can be treated as a single case. Again, relying
 on \eqref{eq:tan_bi_Lip} gives
 \begin{equation}
 \label{eq:ab_lower}
  \sphericalangle( A_\iii(a) , A_\iii(b) )
  \geq
  \frac{1}{M}\frac{ \|A_\iii \eta_2(\iii) \| }{ \| A_\iii \eta_1(\iii) \| } \left| \frac{ s }{ t } - \frac{ v }{ u } \right|
  =
  \frac{1}{M}\frac{\alpha_2(\iii)}{\alpha_1(\iii)} \left| \frac{ s }{ t } - \frac{ v }{ u } \right|
  \geq
  \frac{1}{M}\frac{\alpha_2(\iii)}{\alpha_1(\iii)} \sphericalangle( a , b ).
 \end{equation}
 Combining \eqref{eq:ab_upper1}, \eqref{eq:ab_upper2}, and \eqref{eq:ab_lower} gives
 \begin{equation}
 \label{eq:distortion}
 M^{-1} \frac{\alpha_2(\iii)}{\alpha_1(\iii)}\sphericalangle( a , b )
  \leq
  \sphericalangle( A_\iii(a) , A_\iii(b) )
  \leq
  \frac{\alpha_2(\iii)}{\alpha_1(\iii)} M^2 \sphericalangle( a , b )
 \end{equation}
 and this proves the bounded distortion with the constants $k_0$ and $M^3$.
\end{proof}

If there are not too many cylinders pointing to the same direction then it is possible to get rid of the exceptional directions, but this only works for Hausdorff dimension, since the argument uses countable stability.

\begin{corollary}
 \label{cor:unique_visible}
 Let $E$ be a self-affine satisfying the projection condition and the dominated splitting. Assume further that the sets $\{\pi \iii : \vartheta_1(\iii) = \langle e \rangle \}$ have Hausdorff dimension at most $1$ for all $e\in S^1$. Then $\dimh \vis^e(E) = 1$ for all $e\in S^1$.
\end{corollary}
\begin{proof}
 Assume that $\vartheta_1(\iii) = \langle e \rangle$ for some $\iii\in\Sigma$, since otherwise the claim for $e$ follows from Theorem \ref{thm:general_dominated_vizibility}.
 
 Divide the cylinders of $E$ into different classes according to the angle that the orientation of the cylinder has with $\langle e \rangle$. Set $I(\delta,k) = \{ \iii \in \Sigma^k : \sphericalangle( \vartheta_1(\iii\jjj) , \langle e \rangle ) > \delta \text{ for all } \jjj \in \Sigma \}$ and
 \begin{equation}
  \label{eq:delta_decomposition}
  E(\delta,k) = \bigcup_{\iii\in I(\delta,k)} \fii_{\iii}(E)
 \end{equation}
 From Lemma $\ref{lem:directions}$, it follows that $E = \bigcup_{k\in\N} E(k^{-1},k) \cup F_e$, where $F_e = \{\pi \iii : \vartheta_1(\iii) = \langle e \rangle \}$. It also follows that all the elements of the union are compact sets. The sets $E(k^{-1},k)$ are not exactly self-affine but each of them is a finite union of affine images of $E$. In particular, if $\iii\in I(k^{-1},k)$, then $\langle e \rangle \not\in \vartheta_1(\iii\jjj)$ for all $\jjj\in\Sigma$ and, by Lemma \ref{lem:directions}, $\langle A_\iii^{-1} e \rangle \not\in \vartheta_1(\Sigma)$. Thus Theorem \ref{thm:general_dominated_vizibility} gives that $\dimh \vis^{A_\iii^{-1}e} E = 1$. Recalling that, $\vis^e \fii_\iii(E)$ is an affine image of $\vis^{A_\iii^{-1}e} E$, and that Hausdorff dimension is countably stable finishes the proof.
\end{proof}

\section{Final remarks}
\label{sec:final}
This final section exhibits a few examples dealing with the sharpness of Theorem \ref{thm:general_dominated_vizibility}.
\begin{example}
\label{ex:sharp}
 Consider $f_i \colon [0,1]^2 \to [0,1]^2$ for $i=1,2,3$ with $f_1(x,y) = (3^{-1} x,2^{-1} y)$, $f_2(x,y) = (3^{-1} x,2^{-1} y) + (3^{-1},2^{-1})$, and $f_3(x,y) = (3^{-1} x,2^{-1} y) + (2\cdot 3^{-1},0)$. The  associated self-affine set $E$ is a Bedford-McMullen carpet, and it is well known that
 \[
  \dimh E = \log_2 \left( 2^{\log_3 2} + 1^{\log_3 2}  \right) = \log_2 \left( 2^{\log_3 2} + 1 \right)
 \]
 and
 \[
  \dima E = \log_2 2 + \log_3 2 = 1 + \log_3 2,
 \]
 see for example \cite{Mackay}. In particular, $1 < \dimh E < \dima E$. It is easy to see that the system is dominated and that $\vartheta_1(\Sigma) = \left< \frac{\pi}{2} \right>$.
 
 To verify the projection condition, fix $e\in S^1$ with $\langle e \rangle \neq \langle \frac{\pi}{2} \rangle$. Note that $A_\iii^{-1} (x,y) = (3^n x, 2^n y )$ for all $\iii\in \Sigma^n$, so $n_0\in\N$ can be fixed so that the angle between the x-axis and $A_\iii^{-1} (\langle e \rangle)$ is small, say smaller than $\pi/8$, for all $\iii\in\Sigma^n$, with $n\geq n_0$. By Remark \ref{proj_cond_rem}, it is now enough to show that if $\ell$ is a line so that $\sphericalangle( \ell , \langle \pi \rangle) \leq \pi/8$, and $\ell \cap K \neq \emptyset$, where $K$ is the convex hull of $E$, then also  $\ell \cap E \neq \emptyset$. So let $\ell$ be such a line. The projection of $K$ to arbitrary direction is not an interval, but since the angle between $\ell$ and $x$-axis is small, it is easy tho see that $\ell \cap \fii_{i_1}(K) \neq \emptyset$ for some $i_1\in \{1,2,3\}$. Since $\fii_{i_1}^{-1}\ell$ has even smaller angle with the $x$-axis, it follows that $\fii_{i_1}^{-1}\ell \cap K \neq \emptyset$ implies $\fii_{i_1}^{-1}\ell \cap \fii_{i_2} (K) \neq \emptyset$ for some $i_2 \in \{1,2,3\}$. Continuing in this manner yields a word $\iii = (i_1,i_2,\ldots)\in\Sigma$ so that $\ell \cap \fii_{\iii|_n}(K)$ for all $n\in\N$. Thus $\pi\iii \in \ell \cap E$, and so the projection condition holds.
 
 Since domination and projection condition are satisfied Theorem \ref{thm:general_dominated_vizibility} gives that
 \[
  1 = \dimh \vis^e E = \dima \vis^e E
 \]
 for all $e\neq \pm \frac{\pi}{2}$. However, it is also easy to see that, outside the tri-adic points on the $x$-axis, there is no vertical alignment of points of $E$. Moreover, for each the tri-adic point $t\in[0,1]$, there are at most two points, say $x$ and $y$, of $E$ with $\proj^{ \frac{\pi}{2} } x = t = \proj^{ \frac{\pi}{2} } y$. Therefore, there exist a countable set $H$, so that $H \cup \vis^{\pm \frac{\pi}{2} } E = E$. Thus $ 1 < \dimh E = \dimh \vis^{\pm \frac{\pi}{2} } E $.
 
 The weak tangent $T$ of $E$ that satisfies $\dimh T = \dima E$ is obviously $C \times [0,1]$, where $C$ is the middle thirds Cantor set. However, $\vis^{-\frac{\pi}{2}}(T) = C \times \{ 0 \}$ and so $\dimh \vis^{-\frac{\pi}{2}} T = \log_3 2$. Therefore, if $W$ is the weak tangent of $\vis^{ -\frac{\pi}{2} } E$ that has maximal dimension, then $W \neq \vis^{ -\frac{\pi}{2} } T$.
 
\end{example}

In general the visibility conjecture is false for the Assouad and box dimensions as mentioned in the introduction. The following is a concrete counterexample.

\begin{example}
 \label{ex:assouad_and_box}
 Let $A = \{0\} \cup \{S_n^{-1}\}_{n=1}^\infty$, where $S_n = \sum_{k=1}^n 1/k$. Consider firs the box dimension of just $A \subset \R$. For $\delta>0$, consired the index $n$ for which $\delta_n := (S_n)^{-1} - (S_{n+1})^{-1}$ is closest to $\delta$. Then, to cover $A$ with intervals of length $\delta$, it is essentially enough cover all of $[0,S_n^{-1}]$ and the rest can be neglected. Anyway, at least $N(\delta) \approx \delta_n^{-1} S_n^{-1}$ intervals are needed. On the other hand,
 \begin{align*}
  \delta_n = (S_n)^{-1} - (S_{n+1})^{-1}
  &=
  \frac{(n+1)^{-1}}{S_n S_{n+1}}
  \approx
  \frac{1}{n S_n^2}
 \end{align*}
 and it is an exercise to show that $S_n \approx \log n$. Thus it follows that
 \begin{align*}
  \lim_{\delta \to 0} \frac{\log N(\delta)}{-\log \delta}
  &=
  \lim_{n \to \infty} \frac{\log (\delta_n^{-1} S_n^{-1}) }{\log \delta_n^{-1}}
  =
  \lim_{n \to \infty} 1 - \frac{\log S_n}{\log n + 2\log S_n } = 1,
 \end{align*}
 which implies that $\dimb A = 1$. If one considers $K = A\times A$, essentially the same calculation shows that $\dimb K = 2$. Because $K$ is countable, $\vis^e(K) = K$ for almost all directions, and thus $\dimb \vis^e(K) = 2$ for almost all $e\in S^1$. (For dimensions $d>2$, one can of course consider $K = A^d$, the $d$ fold product of $A$.)
\end{example}

%
%

\bibliographystyle{abbrv}
\bibliography{Bibliography}

\begin{thebibliography}{10}

\bibitem{ArhosaloEtAl2012}
I.~Arhosalo, E.~J{\"a}rvenp{\"a}{\"a}, M.~J{\"a}rvenp{\"a}{\"a}, M.~Rams, and
  P.~Shmerkin.
\newblock {Visible parts of fractal percolation}.
\newblock {\em Proc. Edinb. Math. Soc. (2)}, 55(2):311--331, 2012.

\bibitem{BandtKaenmaki2013}
C.~Bandt and A.~K{\"a}enm{\"a}ki.
\newblock {Local structure of self-affine sets}.
\newblock {\em Ergodic Theory Dynam. Systems}, 33(5):1326--1337, 2013.

\bibitem{BochiGourmelon2009}
J.~Bochi and N.~Gourmelon.
\newblock {Some characterizations of domination}.
\newblock {\em Math. Z.}, 263(1):221--231, 2009.

\bibitem{BondLabaZahl2016}
M.~Bond, I.~{\L}aba, and J.~Zahl.
\newblock {Quantitative visibility estimates for unrectifiable sets in the
  plane.}
\newblock {\em Trans. Amer. Math. Soc.}, 368(8):5475--5513, 2016.

\bibitem{Csornyei2000}
M.~Cs{\"o}rnyei.
\newblock {On the visibility of invisible sets.}
\newblock {\em Ann. Acad. Sci. Fenn. Math.}, 25(2):417--421, 2000.

\bibitem{DaviesFast1978}
R.~O. Davies and H.~Fast.
\newblock {Lebesgue density influences {H}ausdorff measure; large sets
  surface-like from many directions.}
\newblock {\em Mathematika}, 25(1):116--119, 1978.

\bibitem{FalconerFraser2013}
K.~J. Falconer and J.~M. Fraser.
\newblock {The visible part of plane self-similar sets}.
\newblock {\em Proc. Amer. Math. Soc.}, 141(1):269--278, 2013.

\bibitem{JarvenpaaEtAl2003}
E.~J{\"a}rvenp{\"a}{\"a}, M.~J{\"a}rvenp{\"a}{\"a}, P.~MacManus, and T.~C.
  O'Neil.
\newblock {Visible parts and dimensions}.
\newblock {\em Nonlinearity}, 16(3):803--818, 2003.

\bibitem{JarvenpaatSuomalaWu2021}
E.~J{\"a}rvenp{\"a}{\"a}, M.~J{\"a}rvenp{\"a}{\"a}, V.~Suomala, and M.~Wu.
\newblock {On dimensions of visible parts of self-similar sets with finite
  rotation groups}.
\newblock arXiv:2101.01017, 2021.

\bibitem{KaenmakiKoivusaloRossi2017}
A.~K{\"a}enm{\"a}ki, H.~Koivusalo, and E.~Rossi.
\newblock {Self-affine sets with fibred tangents}.
\newblock {\em Ergodic Theory Dynam. Systems}, 37(6):1915--1934, 2017.

\bibitem{KOR}
A.~K{\"a}enm{\"a}ki, T.~Ojala, and E.~Rossi.
\newblock {Rigidity of quasisymmetric mappings on self-affine carpets}.
\newblock {\em Int. Math. Res. Not. IMRN}, (12):3769--3799, 2018.

\bibitem{KaenmakiRossi2016}
A.~K{\"a}enm{\"a}ki and E.~Rossi.
\newblock {Weak separation condition, {A}ssouad dimension, and {F}urstenberg
  homogeneity}.
\newblock {\em Ann. Acad. Sci. Fenn. Math.}, 41(1):465--490, 2016.

\bibitem{KaenmakiShmerkin2009}
A.~K{\"a}enm{\"a}ki and P.~Shmerkin.
\newblock {Overlapping self-affine sets of {K}akeya type}.
\newblock {\em Ergodic Theory Dynam. Systems}, 29(3):941--965, 2009.

\bibitem{Luukkainen1998}
J.~Luukkainen.
\newblock {Assouad dimension: antifractal metrization, porous sets, and
  homogeneous measures}.
\newblock {\em J. Korean Math. Soc.}, 35(1):23--76, 1998.

\bibitem{Mackay}
J.~M. Mackay.
\newblock {Assouad dimension of self-affine carpets}.
\newblock {\em Conform. Geom. Dyn.}, 15:177--187, 2011.

\bibitem{Marstrand1954}
J.~M. Marstrand.
\newblock {Some fundamental geometrical properties of plane sets of fractional
  dimensions.}
\newblock {\em Proc. Lond. Math. Soc. (3)}, 4:257--302, 1954.

\bibitem{Mattila1995}
P.~Mattila.
\newblock {\em {Geometry of sets and measures in {E}uclidean spaces}},
  volume~44 of {\em {Cambridge Studies in Advanced Mathematics}}.
\newblock Cambridge University Press, Cambridge, 1995.
\newblock Fractals and rectifiability.

\bibitem{ONeil2007}
T.~C. O'Neil.
\newblock {The {H}ausdorff dimension of visible sets of planar continua}.
\newblock {\em Trans. Amer. Math. Soc.}, 359(11):5141--5170, 2007.

\bibitem{Orponen2014}
T.~Orponen.
\newblock {Slicing sets and measures, and the dimension of exceptional
  parameters}.
\newblock {\em J. Geom. Anal.}, 24(1):47--80, 2014.

\bibitem{Orponen2019vis}
T.~Orponen.
\newblock {On the dimension of visible parts}.
\newblock {\em J. Eur. Math. Soc.(JEMS)}, 2019.
\newblock To appear.

\bibitem{Orponen2021proj}
T.~Orponen.
\newblock {On the Assouad dimension of projections}.
\newblock {\em Proc. Lond. Math. Soc.(3)}, 122(2):317--351, 2021.

\bibitem{RossiShmerkin2019}
E.~Rossi and P.~Shmerkin.
\newblock {H{\"o}lder coverings of sets of small dimension}.
\newblock {\em J. Fractal Geom.}, 6(3):285--299, 2019.

\bibitem{SimonSolomyak2006}
K.~Simon and B.~Solomyak.
\newblock {Visibility for self-similar sets of dimension one in the plane.}
\newblock {\em Real Anal. Exchange}, 32(1):67--78, 2006.

\end{thebibliography}
\end{document}